\documentclass[10pt]{amsart}
\usepackage{amsmath}
\usepackage[usenames,dvipsnames]{color}
\usepackage{parskip}
\usepackage{amsfonts}
\usepackage{amscd}
\usepackage[centertags]{amsmath}
\usepackage{amssymb}
\usepackage{stmaryrd}
\usepackage[all,cmtip]{xy}
\usepackage[english]{babel}
\usepackage{tabularx}
\usepackage{amsxtra}
\usepackage{euscript}
\usepackage[T1]{fontenc}
\usepackage{doc, exscale, fontenc, latexsym, syntonly}
\usepackage{amsfonts}
\usepackage{amsthm}
\usepackage{graphicx}
\usepackage{mciteplus}
\usepackage{cite}

\newtheorem{theorem}{Theorem}[section]
\newtheorem{corollary}[theorem]{Corollary}
\newtheorem{lemma}[theorem]{Lemma}
\newtheorem{proposition}[theorem]{Proposition}
\newtheorem{remark}[theorem]{Remark}
\theoremstyle{definition}
\newtheorem{definition}[theorem]{Definition}
\newtheorem{example}[theorem]{Example}
\theoremstyle{remark}

\numberwithin{figure}{section}
\numberwithin{table}{section}

\newcommand*\acknowledgment[1]{%
	\begingroup\noindent
	\rightskip\leftskip
	\begin{flushleft}\textbf{\large Acknowledgment.}\, #1%
		\par\vspace*{1mm}\end{flushleft}\endgroup}
\begin{document}

\title[Digital Topological Complexity of Digital Maps]{Digital Topological Complexity of Digital Maps}

\author{MEL\.{I}H \.{I}S and \.{I}SMET KARACA}
\date{\today}

\address{\textsc{Melih Is}
Ege University\\
Faculty of Sciences\\
Department of Mathematics\\
Izmir, Turkey}
\email{melih.is@ege.edu.tr}
\address{\textsc{Ismet Karaca}
Ege University\\
Faculty of Science\\
Department of Mathematics\\
Izmir, Turkey}
\email{ismet.karaca@ege.edu.tr}



\begin{abstract}
In this study, we improve the topological complexity computations on digital images with introducing the digital topological complexity computations of a surjective and digitally continuous map between digital images. We also reveal differences in topological complexity computations of maps between digital images and topological spaces. Moreover, we emphasize the importance of the adjacency relation on the domain and the range of a digital map in these computations.
\end{abstract}

\maketitle

\section{Introduction}
\label{intro}
\quad The study of the topological complexity by topological methods has made great progress since it was first introduced by M.Farber \cite{Farber:2003}. The close link between the topological complexity (denoted by TC) and the Lusternik-Schnirelmann category (denoted by cat) has a positive effect on these studies (see \cite{Farber:2008} for general knowledge about the notion topological complexity). Just as it is usual to talk about the LS-category of a topological space, it is also possible to talk about the LS-category of a map between topological spaces (a detailed analysis of the Lusternik-Schnirelman category of a space or a map is in \cite{CorneaLupOpTan:2003}). Hence, it is a natural result to suggest TC of a surjective and continuous map \cite{Pavesic:2019}. Note that TC of a topological space coincides with TC of a map when the surjective and continuous map is identity. From this fact, TC$(Y)$ is the spesific case of TC$(g)$. Let $g : Y \rightarrow Z$ be surjective and continuous. Then we often compute each of TC$(Y)$, cat$(Y)$, $TC(Z)$, cat$(Z)$ and cat$(g)$. Under some different conditions, all of them can be useful for having the exact number of TC$(g)$. When $g$ is a fibration, TC$(g)$ is more familiar for us because the definition of TC$(g)$ can be given by using the notion Schwarz genus \cite{Schwarz:1966}. 

\quad With the inclusion of digital topology \cite{KaracaIs:2018}, studies on topological complexity start to get different results on digital images compared to topological spaces. As the concept of the topological complexity evolves in topological spaces (for example, the introduction of the notion higher topological complexity, denoted by TC$_{n}$, \cite{Rudyak:2010}), interesting properties are obtained on digital images \cite{MelihKaraca2:2020,MelihKaraca3:2020,MelihKaraca4:2020}. After given the important results on cat$_{\kappa,\lambda}(g)$ by \cite{TaneAyse:2020}, the main idea of this study is to introduce the topological complexity of a digital map and to state its related properties. At this point, it is essential to recall that the digital topology does not have a topology structure. More specifically, \text{a digital image} \cite{Boxer:1999} is a couple $(Y,\kappa)$ in $\mathbb{Z}^{n}$ such that $Y \subset \mathbb{Z}^{n}$ is a set of points with the adjacency relation $\kappa$ on itself. This construction is quite different from a topological space and this gives us that the topological complexity computations in digital images are generally unlike those previously computed in topological spaces.  Adjacency is a relation and given on a digital image $Y$: Let $l$ be a positive integer not exceeding $n$. Let $u = (u_{1}, ..., u_{n})$ and $v = (v_{1}, ..., v_{n})$ be any two different elements in $Y$. $u$ and $v$ are \textit{$c_{l}-$adjacent} \cite{Boxer:1999} if we have $|u_{j} - v_{j}| = 1$ for at most $l$ indices $j$, and $|u_{k} - v_{k}| \neq 1$ implies $u_{k} = v_{k}$ for all indices $k$. It is denoted $u \leftrightarrow_{c_{l}} v$ by $u$ is adjacent to $v$. Another notation $u \leftrightarroweq_{c_{l}} v$ is also used and it means that either $u = v$ or $u \leftrightarrow_{c_{l}} v$. In general, we have $n$ distinct adjacencies in $\mathbb{Z}^{n}$. As an example, one adjacency occurs in $\mathbb{Z}$ and this is $c_{1} = 2-$adjacency. In $\mathbb{Z}^{2}$ and $\mathbb{Z}^{3}$, we have two and three adjacencies, respectively. They are $c_{1} = 4$ and $c_{2} = 8$ adjacencies for $n=2$, and $c_{1} = 6$, $c_{2} = 18$ and $c_{3} = 26$ adjacencies for $n=3$.

\quad One of the biggest difference on the topological complexity (similarly the Lusternik-Schnirelmann category) between digital images and topological spaces is that the topological complexity (similarly the Lusternik-Schnirelmann category) is depend on adjacency relations in digital images. For example, let 
\begin{eqnarray*}
	&&Y = \{(1,-1),(1,0),(1,1),(0,1),(-1,1),(-1,0),(-1,-1),(0,-1)\} \subset \mathbb{Z}^{2}
\end{eqnarray*}
be an image. Then TC$(Y,4)$ is $2$ \cite{MelihKaraca2:2020}. However, TC$(Y,8)$ is $1$ since $Y$ is $8-$contractible. For the LS-category of a digital map, the result does not change, i.e, the digital LS-category of a digital map depends on the adjacency relations on the domain and the range of a digital map. In this paper, we show that this result is valid for not only the Lusternik-Schnirelmann category but also the digital topological complexity. Explicitly, the adjacency relation on the domain and the range of a digital map has a power to change the result for TC.

\quad We define TC of a surjective and digitally continuous map by using the digital version of Schwarz genus in this study. Then we have certain computations of TC of digital maps. Using these computations, we give some counterexamples in digital images. Later, we completely explain how the adjacency relations on the domain and the range of a digital map effect the results. In addition, we prove that TC$(g)$ is a fiber homotopy equivalent of digital images.

\section{Preliminaries}
\label{sec:1}
\quad The basic notions and results on digital images is stated in this section. Hereafter, we shortly write $Y$ and $Z$ instead of the digital images $(Y,\kappa)$ and $(Z,\lambda)$. Similarly we write $g$ instead of the digital map $g : (Y,\kappa) \rightarrow (Z,\lambda)$, unless otherwise stated. The section also contains main properties of topological robotics facts in the digital context. 

\quad To compute the topological complexity, a topological space must be path-connected. Unlike topological spaces, the digital version of the notion path-connectedness corresponds to the notion digital connectedness. Let $Y$ be a digital image such that $U = \{u_{0},u_{1},...,u_{r}\}$ is a subset of $Y$. Then $U$ is a \textit{$\kappa-$path} \cite{Boxer:2006} from $u_{0}$ to $u_{r}$ if $u_{i} \leftrightarroweq_{\kappa} u_{i+1}$ for $i = 0,1,\cdots,r-1$. A digital image $Y$ is \textit{$\kappa-$connected} \cite{Ros:1986} if there is a digital $\kappa-$path between any two points in $Y$. A digital map $g$ is \textit{$(\kappa,\lambda)-$continuous} \cite{Boxer:1999,Ros:1986} if and only if $g(u) \leftrightarroweq_{\lambda} g(v)$ whenever $u \leftrightarrow_{\kappa} v$ in $Y$. The notion homeomorphism transforms to the notion digital isomorphism in digital images and this adaptation has no any different concept, i.e., a digital map $g$ is \textit{$(\kappa,\lambda)-$isomorphism} \cite{Boxer2:2006} if $g$ is bijective, $g$ is digitally $(\kappa,\lambda)-$continuous and the inverse of $g$ is digitally $(\lambda,\kappa)-$continuous.

\quad For determining TC of a cartesian product image, it is needed to define an adjacency relation on this image. Let $Y$ and $Z$ be any digital images and $(u_{1},v_{1})$ and $(u_{2},v_{2})$ be two elements in $Y \times Z$. Then it is said that \textit{$(u_{1},v_{1})$ and $(u_{2},v_{2})$ are adjacent in $Y \times Z$} \cite{BoxKar:2012} if one of the four cases $u_{1} = u_{2}$ and $v_{1} = v_{2}$; or $u_{1} = u_{2}$ and $v_{1} \leftrightarrow_{\lambda} v_{2}$; or $u_{1} \leftrightarrow_{\kappa} u_{2}$ and $v_{1} = v_{2}$; or $u_{1} \leftrightarrow_{\kappa} u_{2}$ and $v_{1} \leftrightarrow_{\lambda} v_{2}$ holds. Let $c$ and $d$ be two distinct points in $\mathbb{Z}$. A \textit{digital interval} \cite{Boxer:2006} from $c$ to $d$ is given by the set $[c,d]_{\mathbb{Z}} = \{y \in \mathbb{Z} : c \leq y \leq d\}$. It is clear that the digital interval has only $2-$adjacency. Using the digital interval, we restate the digital path definition. Let $Y$ be a digital image and let $c$, $d \in Y$. Then a digitally $(2,\kappa)-$continuous map $g : [0,m]_{\mathbb{Z}} \rightarrow Y$ with $g(0) = c$ and $g(m) = d$ is said to be a \textit{$\kappa$-path} \cite{Boxer:2006} from $c$ to $d$. Let $u_{i}$ be a point of $Y$ for $i = 0,1,2,...,l$. Then we write $u_{0}u_{1}u_{2}\cdots u_{l-1}u_{l}$ to mean that we have a path of digital images from $u_{0}$ to $u_{l}$ in $Y$ in which the route of the path is in the given order. 

\quad Homotopy theory is extensively used in topologic robotics studies. As an example, TC$(Y,\kappa)$ is $1$ whenever $Y$ is $\kappa-$contractible. For this purpose, we recall the definition of digital homotopy and some concepts related to it. Let $[0,m]_{\mathbb{Z}}$ be any digital interval and let $g$, $h:(Y,\kappa) \rightarrow (Z,\lambda)$ be any two digitally continuous maps. Then they are \textit{$(\kappa,\lambda)-$homotopic} \cite{Boxer:1999} if for all $u \in Y$, there is a digital map $G : Y \times [0,m]_{\mathbb{Z}} \rightarrow Z$ with $G(u,0) = g(u)$ and $G(u,m) = h(u)$, for any fixed $t \in [0,m]_{\mathbb{Z}}$, the digital map $G_{t} : Y \rightarrow Z$ is digitally $(\kappa,\lambda)-$continuous and for any fixed $u \in Y$, the digital map $G_{u}:[0,m]_{\mathbb{Z}} \rightarrow Z$ is digitally
$(2,\lambda)-$continuous. Note that $m$ is the step number of the homotopy in the definition. One can also say that in the digital meaning, $g$ is homotopic to $h$ in $m$ steps. Consider any digitally continuous map $g$. Then it is a \textit{$(\kappa,\lambda)-$homotopy equivalence} \cite{Boxer:2005} if there is digitally continuous $h : Z \rightarrow Y$ satisfies the following conditions: $h \circ g$ and $1_{Y}$ are digitally homotopic to each other and $g \circ h$ and $1_{Z}$ are digitally homotopic to each other. Consider any digital image $Y$. It is called \textit{digitally $\kappa-$contractible} \cite{Boxer:1999} if $1_{Y}$ and some digital constant maps in $Y$ are digitally homotopic to each other. Let $U$ be a subset of a digital image $Y$. Then a digitally continuous map $g : Y \rightarrow U$ is \textit{digitally $\kappa-$retraction} \cite{Boxer:1994} if $g(u) = u$ for all $u \in U$. 

\begin{definition}\cite{KaracaVergili:2011} 
	Let $(Y,\kappa)$ and $(T,\rho)$ be any digital images and $Z$ be a $\lambda-$connected digital image. Let $g : Y \rightarrow Z$ be a digitally continuous and surjective map. Assume that the following holds:\\
	\textbf{1)} The digital map $g^{-1}(z) \rightarrow T$ is a digital isomorphism for all $z \in Z$.\\
	\textbf{2)} There is a $\lambda-$connected subset $V$ of $Z$ for all elements in $Z$ such that \linebreak$\varphi : g^{-1}(V) \rightarrow V \times T$ is $(\kappa,\kappa_{*})-$isomorphism with $\pi_{1} \circ \varphi = g$, where $\pi_{1}$ is projection.\\ 
	Then the quadruple $\xi = (Y, g, Z, T)$ is a \textit{digital fiber bundle}.
\end{definition}

\begin{example}\label{e4}
	Let $Y = [0,3]_{\mathbb{Z}}$, $Z = [0,1]_{\mathbb{Z}}$ and $T = [4,5]_{\mathbb{Z}}$. Consider the surjective and digitally continuous map $g : [0,3]_{\mathbb{Z}} \rightarrow [0,1]_{\mathbb{Z}}$ with $g(y) =\begin{cases}
			0, & y \in [0,1]_{\mathbb{Z}}\\
			1, & y \in [2,3]_{\mathbb{Z}}.
		\end{cases}$.
		
		\textbf{1)} Since $g^{-1}(0) = [0,1]_{\mathbb{Z}}$ and $g^{-1}(1) = [2,3]_{\mathbb{Z}}$, the digital maps $\alpha : g^{-1}(0) \rightarrow [4,5]_{\mathbb{Z}}$, $\alpha(x) = x+4$ and $\beta : g^{-1}(1) \rightarrow [4,5]_{\mathbb{Z}}$, $\beta(x) = x+2$ are digital isomorphisms.
		
		\textbf{2)} For $0 \in Z$, choose $V_{1} = \{0\}$. Then $g^{-1}(V_{1}) = [0,1]_{\mathbb{Z}}$. Hence, the digital map $\varphi_{1} : g^{-1}(V_{1}) \rightarrow V_{1} \times [4,5]_{\mathbb{Z}}$ is a digital isomorphism. Similarly, for $1 \in [0,1]_{\mathbb{Z}}$ and $V_{2} = \{1\}$, we have a digital isomorphism $\varphi_{2} : g^{-1}(V_{2}) \rightarrow V_{2} \times [4,5]_{\mathbb{Z}}$, where $g^{-1}(V_{2}) = [2,3]_{\mathbb{Z}}$. As a result, $(Y,g,Z,T)$ is a fiber bundle in the digital sense. 
\end{example}

\quad Let $g$ be a digital map. Then it is a \textit{digital fibration} \cite{EgeKaraca:2017} if it has, in the digital setting, the homotopy lifting property for every digital image. $g^{-1}(z_{0})$ is said to be the digital fiber for $z_{0} \in Z$. If a digital map $g$ is not a digital fibration, then we say that $\widehat{g} : (T,\rho) \rightarrow (Z,\lambda)$ is a \textit{digital fibrational substitute} \cite{MelihKaraca:2020} of $g$ if $\widehat{g}$ is a fibration of digital images with the property that a diagram 
\begin{displaymath}
\xymatrix{
	Y \ar[r]^{k} \ar[dr]_g &
	T \ar@{.>}[d]^{\widehat{g}} \\
	& Z, }
\end{displaymath}
is commutative for a digital homotopy equivalence $k : Y \rightarrow T$.

\begin{definition}\cite{MelihKaraca:2020}
	Let $g$ be a fibration of digital images and $h, k :(T,\rho) \rightarrow (Y,\kappa)$ be two digital maps. Then they are \textit{fiber homotopic in the digital sense} if there exists $G : T \times [0,m]_{\mathbb{Z}} \rightarrow Y$ that is a homotopy for the digital maps $h$ and $k$ with $g \circ G(y,t) = g \circ h(y)$ for $y \in Y$ and $t \in [0,m]_{\mathbb{Z}}$. 
\end{definition}

\begin{definition}\cite{MelihKaraca:2020}\label{d1}
	Two digital fibrations $g_{1} : (Y_{1},\kappa_{1}) \rightarrow (Z,\lambda)$ and \linebreak $g_{2} : (Y_{2},\kappa_{2}) \rightarrow (Z,\lambda)$ are \textit{fiber homotopy equivalent in the digital sense} if there exist digital maps $h : (Y_{1},\kappa_{1}) \rightarrow (Y_{2},\kappa_{2})$ and $k : (Y_{2},\kappa_{2}) \rightarrow (Y_{1},\kappa_{1})$ with $k \circ h$ and $h \circ k$ are digitally fiber homotopic to identity maps.  
\end{definition}

\begin{example}
	Consider the digital fibration $g : [0,3]_{\mathbb{Z}} \rightarrow [0,1]_{\mathbb{Z}}$ given by $g(y) =\begin{cases}
	0, & y \in [0,1]_{\mathbb{Z}}\\
	1, & y \in [2,3]_{\mathbb{Z}}
	\end{cases}$ in Example \ref{e4}. Similarly, $g^{'} : [-3,0]_{\mathbb{Z}} \rightarrow [0,1]_{\mathbb{Z}}$, $g^{'}(y) =\begin{cases}
	0, & y \in [-3,-2]_{\mathbb{Z}}\\
	1, & y \in [-1,0]_{\mathbb{Z}},
	\end{cases}$
	is a digital fibration. Define $h : [0,3]_{\mathbb{Z}} \rightarrow [-3,0]_{\mathbb{Z}}$, $h(x) = x-3$ and $k : [-3,0]_{\mathbb{Z}} \rightarrow [0,3]_{\mathbb{Z}}$, $k(x) = x+3$. Then we have that $k \circ h = 1_{[0,3]_{\mathbb{Z}}}$ and $h \circ k = 1_{[-3,0]_{\mathbb{Z}}}$. This gives that $k \circ h$ and $h \circ k$ is fiber homotopic to the respective identities in the digital setting which means $g$ and $g^{'}$ are digital fiber homotopy equivalent.
\end{example}

\begin{definition}\cite{MelihKaraca:2020}
	Let $g$ be a digital fibration. Then the digital version of \textit{Schwarz genus} of $g$ is the least positive number $l$ such that $W_{1}, W_{2}, ..., W_{l}$ covers $Z$ and for each $1 \leq i \leq l$, we have $(\lambda,\kappa)-$continuous map $s_{i} : (W_{i},\lambda) \rightarrow (Y,\kappa)$ with $g \circ s_{i} = id_{W_{i}}$. The digital Schwarz genus is denoted by $genus_{\kappa,\lambda}(g)$.
\end{definition}

\quad If $g$ is not a digital fibration, then the Schwarz genus of the digital map $g$ is the digital Schwarz genus of the digital fibrational substitute of $g$. Given any digital images $Y$ and $Z$, the \textit{function space of digital images} \cite{LupOpreaScov:2019} $Z^{Y}$ consists of a set of all digitally continuous functions from $Y$ to $Z$ and $Z^{Y}$ has an adjacency relation such that for all $\alpha$, $\beta \in Z^{Y}$ and $u$, $v \in Y$, $u \leftrightarroweq_{\kappa} v$ implies that $\alpha(u) \leftrightarroweq_{\lambda} \beta(v)$. Let $g$ and $h$ be two digital paths in $Y^{[0,m]_{\mathbb{Z}}}$ with $\rho-$adjacency. Then the \textit{paths $g$ and $h$ are digitally $\rho-$connected} \cite{KaracaIs:2018} if for all $t$ steps, they are $\rho-$connected. 

\begin{definition}\cite{MelihKaraca:2020},\cite{KaracaIs:2018}
	Let $Y$ be a $\kappa-$connected digital image. The \textit{digital topological complexity} TC$(Y,\kappa)$ of $Y$ is the Schwarz genus of the digital map $\pi$, where \linebreak $\pi : Y^{[0,m]_{\mathbb{Z}}} \rightarrow Y \times Y$, $\pi(\alpha) = (\alpha(0),\alpha(m))$, is a digital fibration for any $\alpha \in Y^{[0,m]_{\mathbb{Z}}}$.
\end{definition}
\quad Note that, in the definition, the number of steps for two digital paths $g$ and $h$ does not have to be the same. Let $g$ and $h$ have $s$ and $t$ steps, respectively. Without lost of generality, we choose $s<t$. Then we increase the number of steps in $g$ by adding the final step of $g$ $(t-s)$ times. Therefore, we have the same number of steps and use the definition of the digitally connectedness between two digital paths.

\begin{proposition}\cite{KaracaIs:2018}
	TC$(Y,\kappa) = 1$ if and only $Y$ is both digitally connected and $\kappa-$contractible.
\end{proposition}

\begin{theorem}\cite{KaracaIs:2018}
	TC$(Y,\kappa)$ is, in the digital setting, an homotopy invariance.
\end{theorem}

\quad Similar to the digital topological complexity, cat of a digital map is a version of the Schwarz genus in the digital setting. Here, the digital fibration is $\pi^{'} : PY \rightarrow Y$, $\pi^{'}(\alpha) = \alpha(1)$, where $PY$ is defined as the set of all digitally continuous digital paths start at $y_{0} \in Y$. Equivalently, one has the following definition:

\begin{definition}\cite{BorVer:2018}
	Let $Y$ be a digitally connected digital image. Then the \textit{LS-category} cat$_{\kappa}(Y)$ of $Y$ is defined as the least number $l$ for which $Y$ has a covering $V_{1}$, $V_{2}$, $\cdots$, $V_{l}$ with the property that each $V_{i}$ is $\kappa-$contractible in $Y$ for $i = 1,2,\cdots,l$.
\end{definition}

\quad Normally, in this definition, the number of elements in the cover of $Y$ is $l+1$. However, we use $l$ for this number, and hence, our results between TC and cat are more consistent. Another Lusternik-Schnirelmann category definition is given on digital maps:

\begin{definition}\cite{TaneAyse:2020}
	Let $g$ be a digital map. Then the \textit{LS-category} cat$_{\kappa,\lambda}(g)$ of $g$ is defined as the least number $l$ for which $Y$ has a covering $V_{1}$, $V_{2}$, $\cdots$, $V_{l}$ with the property that each $g|_{V_{i}}$ is $(\kappa,\lambda)-$homotopic to a constant map from $Y$ to $Z$ for each $i = 1,2,\cdots,l$.
\end{definition}

\begin{proposition}\cite{MelihKaraca2:2020}\label{p2}
	Let $g$ be a digital fibration. Then $genus_{\kappa,\lambda}(g) \leq \text{cat}_{\lambda}(Z)$. In addition, if $Y$ is digitally $\kappa-$contractible, then $genus_{\kappa,\lambda}(g) = \text{cat}_{\lambda}(Z)$. 
\end{proposition}

\section{Topological Complexity of Maps In Digital Images}
\label{sec:2}
\quad Let $Y$ be any digital image and let $m$ be any positive integer. In Definition 3.10 of \cite{MelihKaraca:2020}, an evaluation map of digital images $[0,m]_{\mathbb{Z}}$ and $Z$ is defined as 
 \begin{eqnarray*}
	&&E^{2,\kappa}_{[0,m]_{\mathbb{Z}},Y} : (Y^{[0,m]_{\mathbb{Z}}} \times [0,m]_{\mathbb{Z}},\kappa_{*}) \longrightarrow (Y,\kappa) \\
	&&\hspace*{2.7cm} (\alpha,t) \longmapsto E^{2,\kappa}_{[0,m]_{\mathbb{Z}},Y}(\alpha,t) = \alpha(t),
\end{eqnarray*}
where $Y^{[0,m]_{\mathbb{Z}}} \times [0,m]_{\mathbb{Z}}$ has $\kappa_{*}-$adjacency. To simplify the notations, for $E_{0,Y}(\alpha)$ and $E_{m,Y}(\alpha)$, we understand $E^{2,\kappa}_{[0,m]_{\mathbb{Z}},Y}(\alpha,0)$ and $E^{2,\kappa}_{[0,m]_{\mathbb{Z}},Y}(\alpha,0)$, respectively.
\begin{definition}
	Let $g$ be a digitally continuous and surjective map, where $Y$ and $Z$ are digitally connected images. For any $m \in \mathbb{Z}$, let \[\pi_{g}^{\kappa,\lambda} : (Y^{[0,m]_{\mathbb{Z}}},\kappa_{\ast}) \rightarrow (Y \times Z,\lambda_{\ast})\] be a digital map defined as \[\pi_{g}^{\kappa,\lambda}(\alpha) = (E_{0,Y}(\alpha),g \circ E_{m,Y}(\alpha)) = (id_{Y}^{\kappa,\kappa} \times g^{\kappa,\lambda}) \circ \pi^{2,\kappa}(\alpha),\] where $\pi^{2,\kappa}(\alpha) = (E_{0,Y}(\alpha),E_{m,Y}(\alpha))$. Then the digital topological complexity of the digital map $g$ is defined by \[TC(g)^{\kappa,\lambda} := genus_{\kappa_{\ast},\lambda_{\ast}}(\pi_{g}^{\kappa,\lambda}).\]
\end{definition}

\quad In the definition of digital version of TC$(g)$, we have two cases. First, $g$ can be a fibration of digital images. Since identity map and evaluation map are digital fibrations, $\pi_{g}^{\kappa,\lambda}$ is also a fibration of digital images. If $g$ is not a digital fibration, we use surjective and digitally continuous fibrational substitute of $g$ instead. Now we consider the digital map $g$ as the digital identity map $id_{Y} : Y \rightarrow Y$. Then we get
\begin{eqnarray*}
	\pi_{g}^{\kappa,\lambda} = \pi_{id_{Y}}^{\kappa,\kappa} = (id_{Y} \times id_{Y}) \circ \pi = \pi.
\end{eqnarray*}
As a conclusion, TC$(g)^{\kappa,\kappa} = \text{TC}(Y,\kappa)$ for $g = id_{Y}$. With this result, we have another generalization of TC$(Y,\kappa)$ besides TC$_{n}(Y,\kappa)$.

\begin{example}\label{e1}
	Assume that $g : Y \rightarrow \{z_{0}\}$ is a constant map of digital images (also a fibration in the digital sense). Note that $z_{0}$ does not have to be an element of $Y$. This map is clearly both digitally continuous and surjective. By the definition, we have that 
	\begin{eqnarray*}
		&&\pi_{g}^{\kappa,2} : Y^{[0,m]_{\mathbb{Z}}} \longrightarrow Y \times \{z_{0}\} \\
		&&\hspace*{1.6cm} \alpha \longmapsto \pi_{g}^{\kappa,2}(\alpha) = (\alpha(0),z_{0}).
	\end{eqnarray*}
    Define the digital map
    \begin{eqnarray*}
    	&&s : Y \times \{z_{0}\} \longrightarrow Y^{[0,m]_{\mathbb{Z}}}\\
    	&&\hspace*{1.0cm} (y,z_{0}) \longmapsto s(y,z_{0}) = \epsilon_{y},
    \end{eqnarray*}
    where $\epsilon_{y}$ is the constant digital path at $y$. $s$ is clearly digitally continuous by the definition of adjacency product and we get
    \begin{eqnarray*}
    	\pi_{g}^{\kappa,2} \circ s(y,z_{0}) = \pi_{g}^{\kappa,2}(s(y,z_{0})) = \pi_{g}^{\kappa,2}(\epsilon_{y}) = (y,z_{0}) = id_{Y \times \{z_{0}\}}. 
    \end{eqnarray*}
    As a result, TC$(g)^{\kappa,2} = 1$ when $g$ denotes a digital constant map.
\end{example}

\begin{remark}
	In Example \ref{e1}, if one considers the consant map $g$ with a larger codomain, i.e. for any $z_{0} \in Z$, $g : Y \rightarrow Z$ with $g(y) = z_{0}$, then the result is not valid because of the fact that $g$ is not surjective. 
\end{remark} 

\begin{example} \label{e2}
   Consider the projection map $p_{1} : Y \times Z \rightarrow Y$, $p_{1}(y,z) = y$, of digital images. We first show that $p_{1}$ is a digital fibration. Let $A$ be an arbitrary digital image such that $g : A \rightarrow Y \times Z$ is a digital map with $g(a) = (y,z)$ for any $a \in A$. Let $H : A \times [0,m]_{\mathbb{Z}} \rightarrow Y$ be a digital homotopy. Assume that $p_{1} \circ g = H \circ i$ for an inclusion $i : A \rightarrow A \times [0,m]_{\mathbb{Z}}$. Define the digital homotopy
   \begin{eqnarray*}
   	&&\widetilde{H} : A \times [0,m]_{\mathbb{Z}} \longrightarrow Y \times Z \\
   	&&\hspace*{1.0cm} (a,t) \longmapsto \widetilde{H}(a,t) = (H(a,t),p_{2} \circ g(a)).
   \end{eqnarray*}
   Then we find
   \begin{eqnarray*}
   	\widetilde{H} \circ i(a) = \widetilde{H}(a,t) = (H(a,t),p_{2} \circ g(a)) = (p_{1} \circ g(a),p_{2} \circ g(a)) = g(a)
   \end{eqnarray*}
   and
   \begin{eqnarray*}
   	p_{1} \circ \widetilde{H}(a,t) = p_{1}(H(a,t),p_{2} \circ g(a)) = H(a,t).
   \end{eqnarray*} 
   Thus, $p_{1}$ admits the homotopy lifting property of digital images related to an arbitrary image $A$. For $m \in \mathbb{Z}$, we have the digital map 
   \begin{eqnarray*}
   	&&\pi_{p_{1}}^{\kappa_{\ast},\kappa} : (Y \times Z)^{[0,m]_{\mathbb{Z}}} \longrightarrow (Y \times Z) \times Y\\
   	&&\hspace*{1.6cm} \alpha \longmapsto \pi_{p_{1}}^{\kappa_{\ast},\kappa}(\alpha) = (\alpha(0),p_{1}(\alpha(m))).
   \end{eqnarray*}
   To show that TC$(p_{1})^{\kappa_{\ast},\kappa} = 1$, where $\kappa_{\ast}$ is an adjacency relation on $Y \times Z$, we define the digitally continuous map 
   \begin{eqnarray*}
   	s : (Y \times Z) \times Y \rightarrow (Y \times Z)^{[0,m]_{\mathbb{Z}}}
   \end{eqnarray*}
   via $s((y_{1},z_{1}),y_{2}) = \beta$, where $\beta$ is a digital path in $Y \times Z$ from the initial point $(y_{1},z_{1})$ to the end point $(y_{2},z_{0})$ for a fixed $z_{0} \in Z$. Since we have that \begin{eqnarray*}
   	\pi_{p_{1}}^{\kappa_{\ast},\kappa} \circ s((y_{1},z_{1}),y_{2}) &=& \pi_{p_{1}}^{\kappa_{\ast},\kappa}(s((y_{1},z_{1}),y_{2})) = \pi_{p_{1}}^{\kappa_{\ast},\kappa}(\beta) \\
   	&=& (\beta(0),p_{1}(\beta(m))) = ((y_{1},z_{1}),p_{1}(y_{2},z_{0}))\\
   	&=& ((y_{1},z_{1}),y_{2}) \\ 
   	&=& id_{(Y \times Z) \times Y},
   \end{eqnarray*}
   the desired result holds.
\end{example}

\quad Now consider the digital images $Y \subset \mathbb{Z}^{2}$ with $4-$adjacency and $Z = \{(1,0)\}$, where $Y$ consists of the points
\begin{eqnarray*}
	&&y_{1} = (1,-1), \hspace*{0.3cm} y_{2} = (1,0), \hspace*{0.3cm} y_{3} = (1,1), \hspace*{0.3cm} y_{4} = (0,1),\\ &&y_{5} = (-1,1), \hspace*{0.3cm} y_{6} = (-1,0), \hspace*{0.3cm} y_{7} = (-1,-1), \hspace*{0.3cm} y_{8} = (0,-1).
\end{eqnarray*}
Take the digital projection map $p_{1} : (Y,4) \times (Z,2) \rightarrow (Y,4)$ with $p_{1}(y,z) = y$. We find TC$(p_{1})^{\kappa_{\ast},4} = 1$ for an adjacency relation $\kappa_{\ast}$ on $\mathbb{Z}^{3}$ by using Example \ref{e2}. On the other hand, we have cat$_{4}(Y) = 2$ by Example 2.7 of \cite{MelihKaraca2:2020}. Hence, we find
\begin{eqnarray*}
	\text{TC}(p_{1})^{\kappa_{\ast},4} < \text{cat}_{4}(Y).
\end{eqnarray*} 
This quick result shows that Proposition 3.2 of \cite{Pavesic:2019} is only true for topological spaces, not digital images.

\quad In topological spaces, Pavesic \cite{Pavesic:2019} states that for a map $g : Y \rightarrow Z$, the following holds:
\begin{eqnarray*}
	\max\{\text{cat}(Z),\text{sec}(g)\} \leq \text{TC}(g).
\end{eqnarray*}
Let us show that this statement is not true for digital images. For any digital map $(Z,\lambda)$, consider the digital projection map $p_{1} : (Y \times Z,\lambda_{\ast}) \rightarrow (Y,8)$ with $p_{1}(y,z) = y$, where $Y$ consists of the points
\begin{eqnarray*}
	&&y_{1} = (-2,0), \hspace*{0.3cm} y_{2} = (-1,1), \hspace*{0.3cm} y_{3} = (0,1),\\ &&y_{4} = (1,0), \hspace*{0.3cm} y_{5} = (0,-1), \hspace*{0.3cm} y_{6} = (-1,-1).
\end{eqnarray*} 
Then TC$(p_{1})^{\lambda_{\ast},8} = 1$ from Example \ref{e2}. On the other side, $Y$ is not $8-$contractible. This means that cat$_{8}(Y) > 1$. Using the homotopy invariance property of TC in digital setting, we have that TC numbers of $Y$ and the digital image in Theorem 3.5 of \cite{KaracaIs:2018} are the same. Hence, we get TC$(Y,8) = 2$. By Theorem 5.1 of \cite{KaracaIs:2018}, we find cat$_{8}(Y) = 2$. As a concequence we obtain \begin{eqnarray*}
	\text{TC}(p_{1})^{\lambda_{\ast},8} < \text{cat}_{8}(Y) \leq \max\{\text{cat}_{8}(Y),genus_{\lambda_{\ast},8}(p_{1})\}.
\end{eqnarray*}

\quad Examples can be increased by considering certain digital maps. For example, if we choose $h$ as $(Y,\kappa) \rightarrow \Delta_{Y}(Y)$, $h(y) = (y,y)$, then we similarly obtain that the TC$(h)^{\kappa,\kappa_{*}}$ equals $1$. Note here that if we discuss the diagonal map $\Delta : Y \rightarrow Y \times Y$ of digital images, the computation misleads us since $\Delta$ is not surjective. The next result is an example of the case that TC$(g)$ in digital images is different from one.

\begin{example} \label{e3}
	Let $g : [0,3]_{\mathbb{Z}} \rightarrow [0,1]_{\mathbb{Z}}$ be a piecewise map of digital images defined as
	\begin{eqnarray*}
		g(y) =\begin{cases}
			0, & y \in [0,1]_{\mathbb{Z}}\\
			1, & y \in [2,3]_{\mathbb{Z}}.
	\end{cases}
\end{eqnarray*}
    Let $0 < 1 < 2 < 3$ denote the direction of digital paths in $[0,3]_{\mathbb{Z}}$. By Example \ref{e4}, we have that $g$ is a digital fibration. Now, assume that TC$(g)^{2,2} = 1$. Then there exists a digitally continuous map
    \begin{eqnarray*}
    	s : [0,3]_{\mathbb{Z}} \times [0,1]_{\mathbb{Z}} \rightarrow [0,3]_{\mathbb{Z}}^{[0,m]_{\mathbb{Z}}}
    \end{eqnarray*}
    for $m \in \mathbb{Z}$ such that $\pi_{g}^{2,2} \circ s = id_{[0,3]_{\mathbb{Z}} \times [0,1]_{\mathbb{Z}}}$. Since $\pi_{g}^{2,2}(\varepsilon_{0}) = (0,0) = \pi_{g}^{2,2}(01)$, where $\varepsilon_{0}$ is the constant map at $0$, $s(0,0)$ must be equal to $\varepsilon_{0}$ or $01$. If $s(\varepsilon_{0}) = (0,0)$, then this contradicts the fact that $s$ is digitally continuous because we have that $\pi_{g}^{2,2}(012) = (0,1) = \pi_{g}^{2,2}(0123)$. So we get $s(0,0) = 01$. If this process continues, then we explicitly find 
    \begin{eqnarray*}
    	s : && (0,0) \longmapsto 01\\
    	&& (0,1) \longmapsto 012 \\
    	&& (1,0) \longmapsto \varepsilon_{1} \\
    	&& (1,1) \longmapsto 12 \\
    	&& (2,0) \longmapsto 21 \\
    	&& (2,1) \longmapsto \varepsilon_{2}\\
    	&& (3,0) \longmapsto 321 \\
    	&& (3,1) \longmapsto 32,
    \end{eqnarray*}
    where $\varepsilon_{i}$ for $i = 1,2$ are constant maps at $i$. Therefore, the direction changes for the paths $s(2,0) = 21$, $s(3,0) = 321$ and $s(3,1) = 32$. Hence, we get the result TC$(g)^{2,2} > 1$. Now let 
    \begin{eqnarray*}
    	[0,3]_{\mathbb{Z}} \times [0,1]_{\mathbb{Z}} = A_{1} \cup A_{2}
    \end{eqnarray*}
    for $A_{1} = \{(2,1),(3,0),(3,2)\}$ and $A_{2} = \{(y,z) \in [0,3]_{\mathbb{Z}} \times [0,1]_{\mathbb{Z}} : (y,z) \notin A_{1}\}$. Then for $m_{1}$, $m_{2} \in \mathbb{Z}$, we have that the digital maps
    \begin{eqnarray*}
    	s_{1} : A_{1} \rightarrow [0,2]_{\mathbb{Z}}^{[0,m_{1}]_{\mathbb{Z}}} \hspace*{0.5cm} \text{and} \hspace*{0.5cm} s_{2} : A_{2} \rightarrow [0,2]_{\mathbb{Z}}^{[0,m_{2}]_{\mathbb{Z}}}
    \end{eqnarray*}
    are digitally continuous such that 
    \begin{eqnarray*}
    	\pi_{g}^{2,2} \circ s_{1} = id_{A_{1}} \hspace*{0.5cm} \text{and} \hspace*{0.5cm} \pi_{g}^{2,2} \circ s_{2} = id_{A_{2}}
    \end{eqnarray*}
    holds. Consequently, TC$(g)^{2,2} = 2$. 
    \end{example}

\quad As a result of Example \ref{e3}, Theorem 3.6 of \cite{TaneAyse:2020} is not valid when we consider TC instead of cat. Indeed, TC$(g)^{2,2} = 2$ for $(2,2)-$continuous $g :[0,3]_{\mathbb{Z}} \rightarrow [0,1]_{\mathbb{Z}}$ given by
$g(y) =\begin{cases}
		0, & y \in [0,1]_{\mathbb{Z}}\\
		1, & y \in [2,3]_{\mathbb{Z}}.
\end{cases}$ whereas TC$([0,3]_{\mathbb{Z}},2) = 1$ and TC$([0,1]_{\mathbb{Z}},2) = 1$, this is because $[0,3]_{\mathbb{Z}}$ and $[0,1]_{\mathbb{Z}}$ are digitally contractible. Thus, we conclude that
\begin{eqnarray*}
	\text{TC}(g)^{\kappa,\lambda} \nleq \min\{\text{TC}(Y,\kappa),\text{TC}(Z,\lambda)\}
\end{eqnarray*}
for the digitally continuous map $g : Y \rightarrow Z$ such that $Y$ and $Z$ are digitally connected images.

\begin{theorem}
	Let $r : (Y,\kappa) \rightarrow (U,\kappa)$ be a digital retraction. If TC$^{\kappa,\kappa}(r) = 1$, then $genus_{\kappa,\kappa}(r) = 1$.
\end{theorem}

\begin{proof}
	Let $r : (Y,\kappa) \rightarrow (U,\kappa)$ be a digital retraction. Then we have $r \circ i = id_{(U,\kappa)}$. This means that $r$ is a digitally continuous surjection. Let $genus_{\kappa_{\ast},\lambda_{\ast}}(\pi_{r}) = 1$, where $Y^{[0,m]_{\mathbb{Z}}}$ and $Y \times U$ have $\kappa_{\ast}$ and $\lambda_{\ast}$ adjacencies, respectively. Then $\pi_{r}$ admits a digitally continuous map $s_{1} : V_{1} \rightarrow Y^{[0,m]_{\mathbb{Z}}}$ such that $\pi_{r} \circ s_{1} = id_{Y \times U}$. For a fixed $y_{0} \in Y$, set $\widehat{V_{1}} = \{u \in U : (y_{0},u) \in V_{1}\}$. We define a digitally continuous map
	\begin{eqnarray*}
		&&\widehat{s_{1}} : \widehat{V_{1}} \longrightarrow Y \\
		&&\hspace*{0.7cm} u \longmapsto \widehat{s_{1}}(u) = s_{1}(y_{0},u)(m) = u.
	\end{eqnarray*}
    Hence, we get
    \begin{eqnarray*}
    	r \circ \widehat{s_{1}}(u) = r(u) = u.
    \end{eqnarray*}
    This proves that $genus_{\kappa,\kappa}(r) = 1$.
\end{proof}

\begin{theorem}	
	Let $r : (Y,\kappa) \rightarrow (U,\kappa)$ be a digital retraction. Then \[\text{TC}^{\kappa,\kappa}(r) \geq genus_{\kappa,\kappa}(r).\] 
\end{theorem}

\begin{proof}
	Let $r : Y \rightarrow U$ be a digital retraction and TC$^{\kappa,\kappa}(r) = k$. Then there exists a digital covering $V_{1},...,V_{k}$ of $Y \times U$ such that for each $i \in \{1,...,k\}$, the digitally continuous map $s_{i} : V_{i} \rightarrow Y^{[0,m]_{\mathbb{Z}}}$ satisfies that $\pi_{r} \circ s_{i} = id_{V_{i}}$. For a fixed $y_{0} \in Y$, set $\widehat{V_{i}} = \{u \in U : (y_{0},u) \in V_{i}\}$ for each $i$ and define the inclusion map $h_{i} : \widehat{V_{i}} \rightarrow V_{i}$ with $h_{i}(u) = (y_{0},u)$. Then the composition map $E_{m,Y} \circ s_{i} \circ h_{i} : \widehat{V_{i}} \rightarrow Y$ is a digitally continuous section of the digital image $r$ over each $\widehat{V_{i}}$ such that
	\begin{eqnarray*}
		r \circ (E_{m,Y} \circ s_{i} \circ h_{i})(u) = r \circ (E_{m,Y} \circ s_{i}(y_{0},u)) = r(u) = u = id_{\widehat{V_{i}}}.
	\end{eqnarray*}
    As a result, $genus_{\kappa,\kappa}(r) \leq k$.
\end{proof}

\begin{proposition} \label{p1}
	Let $\kappa$ and $\kappa^{'}$ be adjacency relations on a digital image $Y$ and let $\lambda$ be an adjacency relation on a digital image $Z$. If $\kappa \geq \kappa^{'}$ and $g : (Y,\kappa) \rightarrow (Z,\lambda)$ is $(\kappa,\lambda)-$continuous, then
	\begin{eqnarray*}
		\text{TC}(g)^{\kappa,\lambda} \geq \text{TC}(g)^{\kappa^{'},\lambda}.
	\end{eqnarray*}
\end{proposition}

\begin{proof}
	Let TC$(g)^{\kappa,\lambda} = r$. For any $m \in \mathbb{Z}$, let $\kappa_{\ast}$ be an adjacency relations on $Y^{[0,m]_{\mathbb{Z}}}$. Let $\lambda^{1}_{\ast}$ be a cartesian product adjacency on $Y \times Z$ such that $Y$ has the adjacency $\kappa$ and $\lambda^{2}_{\ast}$ be a cartesian product adjacency on $Y \times Z$ such that $Y$ has the adjacency $\kappa^{'}$. Then Schwarz genus of the map $\pi_{g}^{\kappa,\lambda}$ is $r$. Therefore, we have the partition $Y \times Y = A_{1} \cup A_{2} \cup ... \cup A_{r}$ such that for $i = 1, ..., r$, there exists a digitally $(\lambda^{1}_{\ast},\kappa_{\ast})$-continuous map $s_{i} : A_{i} \rightarrow Y^{[0,m]_{\mathbb{Z}}}$ with $\pi_{g}^{\kappa,\lambda} \circ s_{i} = id_{A_{i}}$. Since $s_{i}$ is $(\lambda^{1}_{\ast},\kappa_{\ast})$-continuous, for any $(y,z)$, $(y^{'},z^{'}) \in A_{i}$ with $(y,z) \leftrightarrow_{\kappa,\lambda} (y^{'},z^{'})$, we have that $s_{i}(y,z) \leftrightarrow_{\kappa_{\ast}} s_{i}(y^{'},z^{'})$. If $\kappa \geq \kappa^{'}$, then $y \leftrightarrow_{\kappa} y^{'}$ implies $y \leftrightarrow_{\kappa^{'}} y^{'}$. By using this fact, we get $s_{i}(y,z) \leftrightarrow_{\kappa_{\ast}} s_{i}(y^{'},z^{'})$ whenever $(y,z) \leftrightarrow_{\kappa^{'},\lambda} (y^{'},z^{'})$. This shows that the digital map $s_{i}$ for all $i$ is digitally $(\lambda^{2}_{\ast},\kappa_{\ast})$-continuous. Hence, we have the partition $Y \times Z = A_{1} \cup A_{2} \cup ... \cup A_{r}$ with the digitally $(\lambda^{2}_{\ast},\kappa_{\ast})-$continuous section $s_{i}$ of $\pi_{g}^{\kappa^{'},\lambda}$ for all $i = 1,...,r$. Finally, we find TC$(g)^{\kappa^{'},\lambda} \leq r$.
\end{proof}

\begin{proposition}
	Let $\lambda$ and $\lambda^{'}$ be adjacency relations on a digital image $Z$ and let $\kappa$ be an adjacency relation on a digital image $Y$. If $\lambda^{'} \geq \lambda$ and $g : (Y,\kappa) \rightarrow (Z,\lambda^{'})$ is $(\kappa,\lambda)-$continuous, then
	\begin{eqnarray*}
		\text{TC}(g)^{\kappa,\lambda} \leq \text{TC}(g)^{\kappa,\lambda^{'}}.
	\end{eqnarray*}
\end{proposition}

\begin{proof}
	The method is similar to the proof of the previous result. Hence, it is enough to say that for $z$, $z^{'} \in Z$, $z \leftrightarrow_{\lambda^{'}} z^{'}$ implies $z \leftrightarrow_{\lambda} z^{'}$ when $\lambda^{'} \geq \lambda$.
\end{proof}

\quad By Example 5.1 of \cite{MelihKaraca:2020}, we know that
Farber's cohomological method is not true for digital images. Theorem 3.19 of \cite{Pavesic:2019} improves this method for TC$(g)$, where $g : Y \rightarrow Z$ is a map and shows that TC$(g) \geq nil(Ker(1,g)^{\ast})$, where $(1,g)^{\ast}$ denotes an induced map between two cohomologies $H^{\ast}(Y \times Z)$ and $H^{\ast}(Y)$. We shall give the digital interpretation of this result. Let $F$ be a field and consider the identitiy $id$ on the digital image $([0,m]_{\mathbb{Z}} \times [0,m]_{\mathbb{Z}},4)$ for any $m \in \mathbb{Z}$. Then the digital contractibility of $[0,m]_{\mathbb{Z}} \times [0,m]_{\mathbb{Z}}$ gives that TC$(id)^{4,4} = 1$. By Example 5.1 of \cite{MelihKaraca:2020}, we have that TC$(id)^{4,4}$ can be less than or equal to the number of the longest nontrivial product in $F$ with considering that $H^{1,4}([0,m]_{\mathbb{Z}} \times [0,m]_{\mathbb{Z}},F) = F$. This proves the following:

\begin{proposition}
	Cohomological cup-product method for TC$(g)$ does not work in digital maps.
\end{proposition}

\begin{lemma}\label{l1}
	If $Y$ is digitally contractible, then $Y^{[0,m]_{\mathbb{Z}}}$ is digitally contractible for any positive integer $m$.
\end{lemma}

\begin{proof}
	Let $Y$ be a digitally contractible image. Then for $m \in \mathbb{Z}$, there exists a homotopy in the digital sense
	\begin{eqnarray*}
		F : Y \times [0,m]_{\mathbb{Z}} \rightarrow Y
	\end{eqnarray*}
    such that $F(y,0) = id_{Y}$ and $F(y,m) = c_{y}$, where $c_{y} : Y \longrightarrow Y$ is defined as $c_{y} = y_{0}$ for a fixed $y_{0} \in Y$. For any $\alpha \in Y^{[0,m]_{\mathbb{Z}}}$, define a digital map $G$ as follows:
    \begin{eqnarray*}
    	G : &&Y^{[0,m]_{\mathbb{Z}}} \times [0,m]_{\mathbb{Z}} \stackrel{E_{I,Y}^{2,\kappa}}{\longrightarrow} Y \stackrel{F_{t}}{\longrightarrow} Y \stackrel{\beta}{\longrightarrow} Y^{[0,m]_{\mathbb{Z}}}\\
    	&&\hspace*{1.2cm}(\alpha,t) \longmapsto \alpha(t) \longmapsto F_{t}(\alpha(t)) \longmapsto \beta,
    \end{eqnarray*}
    where $\beta$ is a digitally continuous digital path in $Y$ that takes $\alpha(t)$ to $F_{t}(\alpha(t))$. Since $E_{I,Y}^{2,\kappa}$, $F_{t}$ and $\beta$ are digitally continuous maps, $G$ is digitally continuous. It is easy to see that $G$ satisfies the homotopy conditions. For $y$, $z \in Y$, let $\alpha$ be a digital path from $y$ to $z$. Then, $G(\alpha,0)$ and $G(\alpha,m)$ are digital paths from $y$ to $y$ and from $z$ to $y_{0}$, respectively. This proves that $Y^{[0,m]_{\mathbb{Z}}}$ is digitally contractible.
\end{proof}

\begin{theorem}
	For a digital fibration $g$, we have that 
    TC$(g)^{\kappa,\lambda} \leq \text{cat}_{\lambda_{\ast}}(Y \times Z)$,
    where $Y \times Z$ has $\lambda_{\ast}-$adjaceny. Moreover, if $Y$ is $\kappa-$contractible, then
    \begin{eqnarray*}
    	\text{TC}(g)^{\kappa,\lambda} = \text{cat}_{\lambda_{\ast}}(Y \times Z).
    \end{eqnarray*}
\end{theorem}

\begin{proof}
	Let $g$ be a digital fibration. Since the evaluation map is a digital fibration, we have that $\pi_{g} = (id_{Y} \times g) \circ \pi$ is a digital fibration. Proposition \ref{p2} gives us $genus_{\kappa_{\ast},\lambda_{\ast}}(\pi_{g}^{\kappa,\lambda}) \leq \text{cat}_{\lambda_{\ast}}(Y \times Z)$,
	where $Y^{[0,m]_{\mathbb{Z}}}$ has $\kappa_{\ast}-$adjacency. This means that TC$(g)^{\kappa,\lambda} \leq \text{cat}_{\lambda_{\ast}}(Y \times Z)$. Now assume that $Y$ is $\kappa-$contractible. By Lemma \ref{l1}, we get $Y^{[0,m]_{\mathbb{Z}}}$ is $\kappa_{\ast}-$contractible. Finally, Proposition \ref{p2} gives us that $genus_{\kappa,\lambda_{\ast}}(\pi_{g}^{\kappa,\lambda}) = \text{cat}_{\lambda_{\ast}}(Y \times Z)$.
\end{proof}

\quad Pavesic \cite{Pavesic:2019} states that TC of a map of topological spaces is a FHE-invariant. We shall show that TC of a digital map is a digital FHE-invariant.

\begin{lemma}\label{l2}
	Let $h : (X,\kappa_{1}) \rightarrow (Y,\kappa_{2})$ and $g : (Y,\kappa_{2}) \rightarrow (Z,\kappa_{3})$ be two digitally continuous and surjective maps.
	
	\textbf{i)} If a continuity of a map $k : (Y,\kappa_{2}) \rightarrow (X,\kappa_{1})$ exists and $h \circ k$ is digitally homotopic to $id_{Y}$, then TC$(g \circ h)^{\kappa_{1},\kappa_{3}} \geq \text{TC}(g)^{\kappa_{2},\kappa_{3}}$.
	
	\textbf{ii)} If a continuity of a map $k : (Y,\kappa_{2}) \rightarrow (X,\kappa_{1})$ exists and $k \circ h$ is digitally homotopic to $id_{X}$, then TC$(g \circ h)^{\kappa_{1},\kappa_{3}} \leq \text{TC}(g)^{\kappa_{2},\kappa_{3}}$.
\end{lemma}

\begin{proof}
	\textbf{i)} Let TC$(g \circ h)^{\kappa_{1},\kappa_{3}} = r$. Then there is a covering $\{V_{1},\cdots,V_{r}\}$ of $X \times Z$ for which there exists digitally continuous $s_{i} : V_{i} \rightarrow X^{[0,m]_{\mathbb{Z}}}$ for any integer $m$ such that $\pi_{g \circ h} \circ s_{i} = id_{U_{i}}$ for each $i = \{1,\cdots,r\}$. Suppose that $H : Y \times [0,m]_{\mathbb{Z}} \rightarrow Y$ is a digital homotopy between $k \circ h$ and identity on $Y$. Theorem 3.18 of \cite{MelihKaraca:2020} states that one takes the digitally continuous map $\overline{H}(y) : Y \rightarrow Y^{[0,m]_{\mathbb{Z}}}$ with $\overline{H}(y)(t) = H(y,1-t)$. For each $i$, define
	\begin{eqnarray*}
		t_{i}(y,z) := \overline{H}(y) \cdot (h \circ s_{i}(k(y),z)). 
	\end{eqnarray*}
    Since $\overline{H}$, $h$ and $k$ are digitally continuous, $t_{i}$ is digitally continuous. Moreover, $t_{i}$ is a digitally continuous section on $(k \times id_{Z})^{-1}(U_{i})$ for each $i$. This shows that TC$(g)^{\kappa_{2},\kappa_{3}} \leq r$.
    
    \textbf{ii)} Similar construction to the first part can be done.
\end{proof}

\begin{corollary}
	For a fibration $g$ of digital images, TC$(g)^{\kappa,\lambda}$ is a FHE-invariant.
\end{corollary}

\begin{proof}
	Let $g : (Y,\kappa) \rightarrow (Z,\lambda)$ and $h : (Y^{'},\kappa^{'}) \rightarrow (Z,\lambda)$ be two digital fibrations such that there exist $k : (Y,\kappa) \rightarrow (Y^{'},\kappa^{'})$ and $l : (Y^{'},\kappa^{'}) \rightarrow (Y,\kappa)$ digital maps for which both $k \circ l$ and $l \circ k$ are digitally homotopic to the identities. Using Lemma \ref{l2}, we finally get
	\begin{eqnarray*}
		\text{TC}(g)^{\kappa,\lambda} = \text{TC}(h \circ k)^{\kappa,\lambda} \geq \text{TC}(h)^{\kappa^{'},\lambda} = \text{TC}(g \circ h)^{\kappa^{'},\lambda} \geq \text{TC}(g)^{\kappa,\lambda}.
	\end{eqnarray*}
\end{proof}

\quad Recall that a digital $n-$sphere \cite{Boxer:2006} is the digital image $[-1,1]_{\mathbb{Z}}^{n+1} \setminus \{0_{n+1}\}$, where $0_{n+1}$ is the origin of $\mathbb{Z}^{n+1}$.

\begin{corollary}
	Let $Y$ be a digital $n-$sphere. Then for the antipodal map of digital images $h : (Y,\kappa) \rightarrow (Y,\kappa)$ with $h(y) = -y$, we have that TC$(h)^{\kappa,\kappa} = \text{TC}(Y,\kappa)$.
\end{corollary} 
    
\begin{proof}
	Since $h \circ h = id_{Y}$, $h$ is a right and left digital homotopy inverse of itself. Hence, Lemma \ref{l2} gives that 
	\begin{eqnarray*}
		\text{TC}(h)^{\kappa,\kappa} \leq \text{TC}(h \circ h)^{\kappa,\kappa} = \text{TC}(id_{Y})^{\kappa,\kappa} = \text{TC}(Y,\kappa)
	\end{eqnarray*}
	and
	\begin{eqnarray*}
		\text{TC}(h)^{\kappa,\kappa} \geq \text{TC}(h \circ h)^{\kappa,\kappa} = \text{TC}(id_{Y})^{\kappa,\kappa} = \text{TC}(Y,\kappa).
	\end{eqnarray*}
	This concludes that TC$(h)^{\kappa,\kappa} = \text{TC}(Y,\kappa)$.
\end{proof}
\section{Conclusion}
\label{sec:3}
\quad TC$(Y,\kappa)$ is improved upon TC$(g)^{\kappa,\lambda}$ for $g$ from $Y$ with $\kappa-$adjacency to $Z$ with $\lambda-$adjacency in this paper. In other words, TC$(Y,\kappa)$ is TC$(g)^{\kappa,\lambda}$ if the map is identity on $Y$. First, we state the definition of TC$(g)^{\kappa,\lambda}$. Then we compare our results with TC or cat of the domain or the range of a digital map. We observe the results by taking different adjacency relations on the domain or the range of a map. With important certain examples, we give some counterexamples in digital images. 

\quad The higher topological complexity TC$_{n}$ of a space \cite{Rudyak:2010} is a general version of a topological complexity of the space. In topological spaces (or digital images), the construction of the (digital) higher topological complexity of a map is still an open problem. The notion (digital) Schwarz genus has not yet been able to provide a solution to the problem. Recalling that TC of a space has more than one definitions, studies may tend to use another methods.  

\acknowledgment{This work was partially supported by Research Fund of the Ege University (Project Number: FDK-2020-21123). In addition, the first author is granted as fellowship by the Scientific and Technological Research Council of Turkey TUBITAK-2211-A.}

\end{document}